\documentclass[11pt,a4]{article}

\usepackage{amsthm,amssymb}
\usepackage{amsmath,yhmath}
\usepackage{mathrsfs}
\usepackage{hyperref,graphicx,color}
\usepackage[margin=3cm]{geometry}

\hypersetup{linkcolor=black,citecolor=black,colorlinks=true}

\theoremstyle{plain}
\newtheorem{theorem}{Theorem}[section]
\newtheorem{corollary}[theorem]{Corollary}
\newtheorem{lemma}[theorem]{Lemma}
\newtheorem{proposition}[theorem]{Proposition}

\theoremstyle{definition}

\theoremstyle{remark}

\begin{document}

\title{Hardy Spaces ($0<p<\infty$) over Lipschitz Domains}
\author{Deng Guantie\thanks{
      E-mail: denggt@bnu.edu.cn, School of Mathematical Sciences, 
      Beijing Normal University, Beijing, China.} 
    and Liu Rong\thanks{
      Corresponding author, E-mail: rong.liu@mail.bnu.edu.cn
      School of Mathematical Sciences, 
      Beijing Normal University, Beijing, China.}
}

\maketitle

\begin{abstract}
  Let $0<p<\infty$, $\Gamma$ be a Lipschitz curve on the 
  complex plane~$\mathbb{C}$ and $\Omega_+$ is the domain above $\Gamma$, 
  we define Hardy space $H^p(\Omega_+)$ as the set of analytic functions $F$ 
  satisfying $\sup_{\tau>0}(\int_{\Gamma} 
    |F(\zeta+\mathrm{i}\tau)|^p |\,\mathrm{d}\zeta|)^{\frac1p}< \infty$.
  We denote the conformal mapping from $\mathbb{C}_+$ onto $\Omega_+$ as $\Phi$, 
  and prove that, $H^p(\Omega_+)$ is isomorphic to $H^p(\mathbb{C}_+)$, 
  the classical Hardy space on the upper half plane~$\mathbb{C}_+$, 
  under the mapping $T\colon F\to F(\Phi)\cdot (\Phi')^{\frac1p}$. 
  Besides, $T$ and $T^{-1}$ are both bounded. We also prove that 
  if $F(w)\in H^p(\Omega_+)$, then $F(w)$ has non-tangential boundary limit 
  $F(\zeta)$ a.e.\@ on $\Gamma$, and, if $1\leqslant p< \infty$, 
  $F(w)$ is the Cauchy integral on $\Gamma$ of $F(\zeta)$. 
\end{abstract}



\section{Introduction}
We will study Hardy space $H^p(\Omega_+)$ 
over Lipschitz domain~$\Omega_+$, where $0<p<\infty$, in this paper. 
The latter is the simply connected domain over a Lipschitz curve~$\Gamma$, 
and the former is defined as the set of functions which are $p$-integrable 
on curves parallel to $\Gamma$ with a uniform boundary. 

There are two more definitions of $H^p(\Omega_+)$, by \cite{Du70}.
One way is to demand the subharmonic function $\lvert f\rvert^p$ has 
a harmonic majorant on $\Omega_+$, and we denote the function space as $\widetilde{H}^p(\Omega_+)$. The other way is to require the boundedness of the integrals of $\lvert f\rvert^p$
over certain curves tending to the boundary $\Gamma$, in the sense that 
they eventually surround every compact subset of $\Omega_+$. We will denote the resulting 
function space as $E^p(\Omega_+)$, for distinction.

By Riemann mapping theorem, there exists a conformal mapping $\Phi$ 
from $\mathbb{C}_+$ onto $\Omega_+$. Duren has proved that: if $a\leqslant \lvert\Phi'(w)\rvert\leqslant b$, 
where $a$, $b$ are two positive constants and $w\in\Omega_+$, then 
$E^p(\Omega_+)$ coincides with $\widetilde{H}^p(\Omega_+)$~\cite{Du70}. 
In our setting, $|\Phi'(z)|$ is, of course, unbounded. Nevertheless,  
we could prove that $H^p(\Omega_+)\subset E^p(\Omega_+)$, 
and that $H^p(\Omega_+)$ is isomorphic to $H^p(\mathbb{C}_+)$ under a bounded 
linear transform which depends on $\Phi$ and whose inverse is also bounded.

In this paper, $p$ is usually set to be in $(0,\infty)$. In our paper~\cite{DL17}, 
the range of $p$ is $(1,\infty)$. 
Thus, this paper extends many results in \cite{DL17}, including  
the existence of non-tangential boundary limit of functions in $H^p(\Omega_+)$ for $0<p<\infty$, and the Cauchy representation of $H^p(\Omega_+)$ functions for $1\leqslant p<\infty$. We also offer new 
proofs to some proven theorems in that paper. 

\section{Basic Definition}
Let $0<p\leqslant\infty$, and for $1\leqslant p\leqslant\infty$, 
denote $q$ as the conjugate coefficient of $p$, 
which means that $\frac{1}{p}+\frac{1}{q}=1$. 
We first introduce definitions of Hardy spaces over Lipschitz domains.

Let $a\colon\mathbb{R}\to\mathbb{R}$ be a Lipschitz function, 
where $|a(u_1)-a(u_2)|\leqslant M|u_1-u_2|$ 
for all $u_1$, $u_2\in\mathbb{R}$ and fixed $M>0$, 
$\Gamma=\{\zeta(u)=u+\mathrm{i}a(u)\colon u\in\mathbb{R}\}\subset\mathbb{C}$ 
be the graph of $a(u)$, we have $|a'(u)|\leqslant M$ a.e.\@, 
and the arc length measure of $\Gamma$ is 
$|\mathrm{d}\zeta|= (1+{a'}^2(u))^{1/2}\,\mathrm{d}u$. 
It follows that 
$\mathrm{d}u\leqslant |\mathrm{d}\zeta|\leqslant (1+M^2)^{1/2}\mathrm{d}u$. 
For $F(\zeta)$ defined on $\Gamma$, since 
$\int_{\Gamma} |F(\zeta)|^p |\mathrm{d}\zeta|
  = \int_{\mathbb{R}} |F(u+\mathrm{i}a(u))|^p
    \cdot |1+\mathrm{i}a'(u)|\,\mathrm{d}u$, we have
\[\int_{\mathbb{R}} \big|F\big(u+\mathrm{i}a(u)\big)\big|^p \,\mathrm{d}u
  \leqslant \int_{\Gamma} |F(\zeta)|^p |\mathrm{d}\zeta|
  \leqslant \sqrt{1+M^2}\int_{\mathbb{R}} 
      \big|F\big(u+\mathrm{i}a(u)\big)\big|^p \,\mathrm{d}u,\]
that is, $F(\zeta)\in L^p(\Gamma,|\mathrm{d}\zeta|)$ if and only if 
$F(u+\mathrm{i}a(u))\in L^p(\mathbb{R},\mathrm{d}u)$.
The space $L^p(\Gamma,|\mathrm{d}\zeta|)$ may thus be identified with 
$L^p(\mathbb{R},\mathrm{d}u)$. We let $\Omega_+$ denote the set 
$\{u+\mathrm{i}v\in\mathbb{C} \colon v>a(u)\}$, $\Omega_-$ denote  
$\{u+\mathrm{i}v\in\mathbb{C} \colon v<a(u)\}$, and $\Gamma_{\tau}$ denote 
$\Gamma+\mathrm{i}\tau= \{\zeta+\mathrm{i}\tau\colon \zeta\in\Gamma\}$ 
for $\tau\in\mathbb{R}$.

Let $F(w)$ be an analytic function on $\Omega_+$, 
we say that $F(w)\in H^p(\Omega_+)$, if
\[\sup_{\tau>0}\Big(\int_{\Gamma_\tau} |F(w)|^p |\mathrm{d}w|\Big)^{\frac1p}
  = \lVert F\rVert_{H^p(\Omega_+)}<\infty,\quad
  \text{for } 0<p<\infty,\]
or in the case of\/ $p=\infty$,
\[\sup_{w\in\Omega_+}|F(w)|=\lVert F\rVert_{H^\infty(\Omega_+)}<\infty.\]
Notice that 
\[\int_{\Gamma_\tau} |F(w)|^p |\mathrm{d}w|
  =\int_{\Gamma} |F(\zeta+\mathrm{i}\tau)|^p |\mathrm{d}\zeta|.\]

Fix $u_0\in\mathbb{R}$ such that 
$\zeta'(u_0)=|\zeta'(u_0)|\mathrm{e}^{\mathrm{i}\phi_0}$ exists, and 
choose $\phi\in(0,\frac{\pi}2)$, we denote 
$\zeta_0=\zeta(u_0)=u_0+\mathrm{i}a(u_0)$ and let
\[\Omega_{\phi}(\zeta_0)
  =\{\zeta_0+r\mathrm{e}^{\mathrm{i}\theta}
     \colon r>0, \theta-\phi_0\in(\phi,\pi-\phi)\},\]
then we say that a function $F(w)$, defined on $\Omega_+$, 
has non-tangential boundary limit~$l$ at $\zeta_0$ if
\[\lim_{\substack{w\in\Omega_{\phi}(\zeta_0)\cap\Omega_+,\\ w\to\zeta_0}}
  F(w)= l, \quad \text{for any } \phi\in\Big(0,\frac{\pi}2\Big).\]
It is easy to verify that for fixed $\phi\in(0,\frac{\pi}2)$, 
there exists constant $\delta>0$, such that, if $|z|<\delta$ and 
$\zeta_0+z\in\Omega_\phi(\zeta_0)$, then $\zeta_0+z\in\Omega_+$ and 
$\zeta_0-z\in\Omega_-$.

We then turn to definitions of the classical Hardy spaces over $\mathbb{C}_+$, 
the upper half complex plane. If $f(z)$ is analytic on 
$\mathbb{C}_+$ and $y>0$, we define
\[m(f,y)=\Big(\int_{\mathbb{R}} \lvert f(x+\mathrm{i}y)\rvert^p 
\,\mathrm{d}x\Big)^{\frac1p},\quad \text{if } 0<p<\infty,\]
or
\[m(f,y)= \sup_{x\in\mathbb{R}}\lvert f(x+\mathrm{i}y)\rvert,\quad 
\text{if } p=\infty,\]
then Hardy space $H^p(\mathbb{C}_+)$ is defined as
\[H^p(\mathbb{C}_+)
= \big\{f(z)\text{ is analytic on }\mathbb{C}_+\colon
\sup_{y>0} m(f,y)<\infty\big\},\]
and for $f(z)\in H^p(\mathbb{C}_+)$, define
\[\lVert f\rVert_{H^p(\mathbb{C}_+)}= \sup_{y>0} m(f,y).\]

Let $x_0\in\mathbb{R}$, $\alpha>0$ and 
$\Gamma_\alpha(x_0)=\{(x,y)\in\mathbb{C}_+\colon |x-x_0|<\alpha y\}$,
we say that a function $f(z)$, defined on $\mathbb{C}_+$, 
has non-tangential boudary limit~$l$ at $x_0$, if
\[\lim_{\substack{z\in\Gamma_\alpha(x_0),\\ z\to x_0}}f(z)=l,\quad
\text{for any }\alpha>0.\]
It is well-known that every function $f(z)\in H^p(\mathbb{C}_+)$, 
where $0<p\leqslant\infty$, has non-tangential boundary limit a.e.\@ 
on the real axis. We usually denote the limit function as $f(x)$ 
for $x\in\mathbb{R}$, then $\lVert f\rVert_{H^p(\mathbb{C}_+)}
  = \lVert f\rVert_{L^p(\mathbb{R})}$.
If $1\leqslant p\leqslant\infty$, then $f(z)\in H^p(\mathbb{C}_+)$ 
is both the Cauchy and Poisson integral of $f(x)$~\cite{De10}, 
that is, 
\[f(z)
  = \frac1{2\pi\mathrm{i}} \int_{\mathbb{R}}\frac{f(t)\,\mathrm{d}t}{t-z}
  = \frac1{\pi} \int_{\mathbb{R}} \frac{yf(t)\,\mathrm{d}t}{(x-t)^2+y^2},\]
for $z=x+\mathrm{i}y\in\mathbb{C}_+$.

By the definitions above, if $a(u)=0$, then $H^p(\Omega_+)= H^p(\mathbb{C}_+)$, 
thus we may consider $H^p(\mathbb{C}_+)$ as a special case of $H^p(\Omega_+)$. 
It is not difficult to verify that $H^p(\Omega_+)$ is a linear vector space 
equipped with the norm $\lVert\cdot\rVert_{H^p(\Omega_+)}$ 
if $1\leqslant p\leqslant\infty$, or with the metric 
$\lVert\cdot\rVert_{H^p(\Omega_+)}^p$ if $0<p<1$. In this paper, 
we will always assume $0<p<\infty$, if not stated otherwise.

Let $\zeta$, $\zeta_0\in\Gamma$, we define
\[K_z(\zeta,\zeta_0)
  = \frac1{2\pi\mathrm{i}}\bigg(\frac1{\zeta-(\zeta_0+z)}
       - \frac1{\zeta-(\zeta_0-z)}\bigg),\]
for $z\in\mathbb{C}$ and $z\neq \pm(\zeta-\zeta_0)$, 
then $K_z(\zeta,\zeta_0)$ is well-defined and we could write 
\begin{equation}\label{equ-170623-1150}
  K_z(\zeta,\zeta_0)
  = \frac1{\pi\mathrm{i}}\cdot\frac{z}{(\zeta-\zeta_0)^2-z^2},
\end{equation}
We could also verify that, if $\zeta_0+z\in\Omega_+$ and 
$\zeta_0-z\in\Omega_-$, then 
\[\int_{\Gamma} K_z(\zeta,\zeta_0)\,\mathrm{d}\zeta= 1.\]

Since $\Omega_+$ is an open and simply-connected subset of the complex plane, 
there exists a conformal representation $\Phi$ from $\mathbb{C}_+$ 
onto $\Omega_+$, which extends to an increasing homeomorphism of 
$\mathbb{R}$ onto $\Gamma$, that is $\mathrm{Re\,}\Phi'(x)>0$, 
for $x\in\mathbb{R}$ a.e.\@. Define 
$\Sigma= \{x+\mathrm{i}y\in\mathbb{C}\colon 
    x>0, \lvert y\rvert\leqslant Mx\}$, then
$\Phi'(x)\in \overline{\Sigma}$ a.e.\@.

Denote the inverse of $\Phi(z)$ as $\Psi(w)\colon \Omega_+\to\mathbb{C}_+$, 
then 
\[\Phi\big(\Psi(w)\big)=w, \quad\text{for all } w\in\Omega_+,\]
and
\[\Psi\big(\Phi(z)\big)=z, \quad\text{for all } z\in\mathbb{C}_+.\]
Besides, if $w=\Phi(z)$ for $z\in\overline{\mathbb{C}_+}$, 
then $\Phi'(z)\cdot\Psi'(w)=1$, 
thus $\mathrm{Re}\,\Psi'(\zeta)>0$ for $\zeta\in\Gamma$ a.e.\@.
More details about $\Phi$ have been proved by Kenig in \cite{Ke90}, 
which are shown below.
\begin{lemma}[\cite{Ke90}]\label{lem-170706-2130}
  {\rm (i)} $\Phi$ extends to $\overline{\mathbb{C}_+}$ as a homeomorphism 
  onto $\overline{\Omega_+}$.
  
  {\rm (ii)} $\Phi'$ has a non-tangential limit almost everywhere, and this limit
  is different from $0$ almost everywhere.
  
  {\rm (iii)} $\Phi$ preserves angles at almost every boundary point, 
  and so does $\Psi$ (here almost every refers to $|\mathrm{d}\zeta|$).
  
  {\rm (iv)} $\Phi(x)$, $x\in\mathbb{R}$, is absolutely continuous 
  when restricted to any finite interval, and hence $\Phi'(x)$ exists 
  almost everywhere, and is locally integrable, moreover, 
  for almost every $x\in\mathbb{R}$, $\Phi'(x) = \lim_{z\to x} \Phi'(z)$,
  where $z\in\mathbb{C}_+$ converges nontangentially to $x$.
  
  {\rm (v)} Sets of measure $0$ on $\mathbb{R}$ correspond to sets 
  of measure $0$ (with respect to $|\mathrm{d}\zeta|$) on $\Gamma$, 
  and vice versa.
  
  {\rm (vi)} At every point where $\Phi'(x)$ exists and is different from $0$, 
  it is a vector tangent to the curve $\Gamma$ at the point $\Phi(x)$, 
  and hence $|\arg\Phi'(x)|\leqslant \arctan M$ for almost every $x$.
\end{lemma}

We now consider a transform $T$ from $H^p(\Omega_+)$ to 
analytic functions on $\mathbb{C}_+$, which is defined by 
\begin{equation}\label{equ-170602-1910}
  TF(z)= F\big(\Phi(z)\big)\cdot (\Phi'(z))^{\frac1p}, \quad
    \text{for } F(w)\in H^p(\Omega_+),
\end{equation}
then, obviously, $T$ is linear and one-to-one. One of the main results 
of this paper is the following theorem, which shows that 
$H^p(\Omega_+)$ is isomorphic to $H^p(\mathbb{C}_+)$.
\begin{theorem}\label{thm-170803-1520}
  If $0<p<\infty$ and $T$ is defined as above, 
  then $T\colon H^p(\Omega_+)\to H^p(\mathbb{C}_+)$ 
  is linear, one-to-one, onto and bounded. Its inverse 
  $T^{-1}\colon H^p(\mathbb{C}_+)\to H^p(\Omega_+)$ is also bounded.
\end{theorem}
We will also prove that every function in $H^p(\Omega_+)$ has 
non-tangential boundary limit a.e.\@ on $\Gamma$, and, if $1\leqslant p< \infty$, 
is the Cauchy integral of its boundary function. 
See Theorem~\ref{thm-170801-1650} and Theorem~\ref{thm-170803-1020} for details.

\section{Proof of the Isomorphic Theorem}

Let $D$ be an arbitrary simply connected domain with at least 
two boundary points, and remeber that $0<p<\infty$ by default.
A function~$f$ analytic on $D$ is said to be of class~$E^p(D)$ 
if there exists a sequence of rectifiable Jordan curves~$C_1$, 
$C_2$, $\ldots$ in $D$, which eventually 
surround each compact subdomain of $D$, such that
\[\sup_{n\geqslant 1}\int_{C_n} |f(z)|^p |\mathrm{d}z|< \infty.\]

\begin{lemma}[\cite{Du70}]\label{lem-170801-1420}
  Let $\phi$ map the unit disk $\mathbb{D}=\{\xi\in\mathbb{C}\colon |\xi|<1\}$ 
  conformally onto $D$, and $\Gamma_r$ be the image under $\phi$ of 
  the circle~$\{|\xi|=r<1\}$. Then for each function~$f\in E^p(D)$,
  \[\sup_{0<r<1} \int_{\Gamma_r} |f(z)|^p|\mathrm{d}z|< \infty.\]
\end{lemma}

The subharmonic function~$v(z)$ on $D$ is said to have a harmonic majorant 
if there is a harmonic function~$u(z)$ such that $v(z)\leqslant u(z)$ 
throughout $D$. 

\begin{lemma}[\cite{De10}]\label{lem-170801-1430}
  If $v$ is subharmonic on $\mathrm{D}$, then $v$ has a harmonic majorant 
  if and only if 
  \[\sup_{0<r<1} \int_0^{2\pi} v(r\mathrm{e}^{\mathrm{i}\theta})
        \,\mathrm{d}\theta< \infty.\]
\end{lemma}

\begin{lemma}[\cite{Ga07}]\label{lem-170801-1440}
  Let $f(z)$ be an analytic function on the upper half plane~$\mathbb{C}_+$, 
  then $|f(z)|^p$ has a harmonic majorant if and only if 
  $f(z)(z+\mathrm{i})^{-\frac2p}\in H^p(\mathbb{C}_+)$.
\end{lemma}

\begin{theorem}\label{thm-170825-1400}
  If $T$ is defined on $E^p(\Omega_+)$ as in \eqref{equ-170602-1910}, 
  then $T(E^p(\Omega_+))\subset H^p(\mathbb{C}_+)$.
\end{theorem}

\begin{proof}  
  Let $T_1(\xi)= -\mathrm{i}(\xi-1)(\xi+1)^{-1}$, $\phi(\xi)= \Phi(T_1(\xi))$, 
  then $T_1$ maps $\mathbb{D}$ conformally onto $\mathbb{C}_+$ and 
  $\phi$ maps $\mathbb{D}$ conformally onto $\Omega_+$. 
  Choose $F(w)\in E^p(\Omega_+)$, by Lemma~\ref{lem-170801-1420},
  \[\sup_{0<r<1} \int_{\Gamma_r} |F(w)|^p|\mathrm{d}w|= M_1< \infty,\]
  where $\Gamma_r= \phi(D(0,r))$. Define function $v=|F(\phi)|^p|\phi'|$ 
  on $\mathbb{D}$, then $v$ is non-negative, subharmonic and continuous, and 
  \begin{align*}
    \int_{\Gamma_r} |F(w)|^p|\mathrm{d}w|
    &= \int_{|\xi|=r} |F(\phi(\xi))|^p |\phi'(\xi)| |\mathrm{d}\xi|       \\
    &= \int_{|\xi|=r} v(\xi) |\mathrm{d}\xi|                          \\
    &= r\int_0^{2\pi} v(r\mathrm{e}^{\mathrm{i}\theta})\,\mathrm{d}\theta 
     \leqslant M_1.
  \end{align*}
  
  For $r\in[\frac12,1)$, we then have
  \begin{equation}\label{equ-170801-1435}
    \int_0^{2\pi} v(r\mathrm{e}^{\mathrm{i}\theta})\,\mathrm{d}\theta 
    \leqslant \frac{M_1}{r}
    \leqslant 2M_1.
  \end{equation}
  Since $v$ is non-negative and continous on $\overline{D(0,\frac12)}$, 
  we define
  \[M_2= \max\Big\{v(\xi)\colon |\xi|\leqslant \frac12\Big\},\]
  then for $r\in(0,\frac12)$,
  \begin{equation}\label{equ-170801-1445}
    \int_0^{2\pi} v(r\mathrm{e}^{\mathrm{i}\theta})\,\mathrm{d}\theta 
    \leqslant 2\pi M_2.
  \end{equation}
  Combining inequalities~\eqref{equ-170801-1435} and \eqref{equ-170801-1445}, 
  we have 
  \[\sup_{0<r<1} \int_0^{2\pi} 
        v(r\mathrm{e}^{\mathrm{i}\theta})\,\mathrm{d}\theta
    \leqslant \max\{2M_1,2\pi M_2\}
    < \infty.\]
  Lemma~\ref{lem-170801-1430} implies that $v$ has a harmonic majorant 
  on $\mathbb{D}$, which means that $|F(\phi)|^p|\phi'|\leqslant u$ 
  for a harmonic function~$u$.
  
  Since $\phi=\Phi(T_1)$, $T_1(\xi)= -\mathrm{i}(\xi-1)(\xi+1)^{-1}$, 
  then $T_1'(\xi)= -2\mathrm{i}(\xi+1)^{-2}$ and 
  \[|F(\Phi(T_1))|^p|\Phi'(T_1)T_1'|\leqslant u\quad
    \text{on } \mathbb{D}.\]
  Let $z= T_1(\xi)\in\mathbb{C}_+$, then 
  $\xi= T_1^{-1}(z)= -(z-\mathrm{i})(z+\mathrm{i})^{-1}$,
  \[T_1'(\xi)= T_1'(T_1^{-1}(z))
    = \frac{\mathrm{i}(z+\mathrm{i})^2}2,\]
  and
  \[|F(\Phi(z))|^p|\Phi'(z)|\cdot \frac{|z+\mathrm{i}|^2}2
    \leqslant u(T_1^{-1}(z)).\]
  Let $f(z)= F(\Phi(z))(\Phi'(z))^{\frac1p}(z+\mathrm{i})^{\frac2p}$ 
  for $z\in\mathbb{C}_+$, since $u(T_1^{-1})$ is harmonic, 
  the above inequality shows that 
  $|f(z)|^p$ has a harmonic majorant. By Lemma~\ref{lem-170801-1440},
  \[TF(z)
    = F(\Phi(z))(\Phi'(z))^{\frac1p}
    = \frac{f(z)}{(z+\mathrm{i})^{\frac2p}}
    \in H^p(\mathbb{C}_+),\]
  which shows that $T(E^p(\Omega_+))\in H^p(\mathbb{C}_+)$.
\end{proof}

\begin{proposition}\label{pro-170801-1410}
  For $T$ defined by \eqref{equ-170602-1910}, 
  $T(H^p(\Omega_+))\subset H^p(\mathbb{C}_+)$.
\end{proposition}

\begin{proof}
  By Theorem~\ref{thm-170825-1400}, we only need to verify that 
  $H^p(\Omega_+)\subset E^p(\Omega_+)$ for all $0<p<\infty$.
  Let $F(w)\in H^p(\Omega_+)$ and $E_n=\{\zeta+\mathrm{i}\tau\colon 
    \zeta\in\Gamma, \tau\in[\frac1n,n]\}$ for positive integer~$n$, then 
    \[\iint_{E_n} |F(w)|^p\,\mathrm{d}\lambda(w)
      = \int_{\frac1n}^n\!\!\int_\Gamma |F(\zeta+\mathrm{i}\tau)|^p 
          \,|\mathrm{d}\zeta|\,\mathrm{d}\tau
      \leqslant n\lVert F\rVert_{H^p(\Omega_+)}^p,\]
  where $\mathrm{d}\lambda$ is the area measure on $\mathbb{C}$.
  Fix $n$, since
  \begin{align*}
    \iint_{E_n} |F(w)|^p\,\mathrm{d}\lambda(w)
    &= \int_\Gamma\!\int_{\frac1n}^n |F(\zeta+\mathrm{i}\tau)|^p 
        \,\mathrm{d}\tau\,|\mathrm{d}\zeta|                         \\
    &\geqslant \int_0^{+\infty}\!\!\!\int_{\frac1n}^n \Big(
        |F(\zeta(u)+\mathrm{i}\tau)|^p+ |F(\zeta(-u)+\mathrm{i}\tau)|^p\Big) 
        \,\mathrm{d}\tau\,\mathrm{d}u,
  \end{align*}
  we have
  \[\lim_{u\to +\infty} \int_{\frac1n}^n \Big(
          |F(\zeta(u)+\mathrm{i}\tau)|^p+ |F(\zeta(-u)+\mathrm{i}\tau)|^p\Big) 
          \,\mathrm{d}\tau
     = 0.\]
  Then we could choose $u_n>n$ such that
  \[\int_{\frac1n}^n \Big(|F(\zeta(u_n)+\mathrm{i}\tau)|^p
        + |F(\zeta(-u_n)+\mathrm{i}\tau)|^p\Big)\,\mathrm{d}\tau< 1,\] 
  and define $C_n$ as the boundary of $\{\zeta(u)+\mathrm{i}\tau\colon 
    |u|\leqslant u_n, \tau\in[\frac1n,n]\}$. It follows that
  \[\int_{C_n} |F(w)|^p |\mathrm{d}w|
    \leqslant 2\lVert F\rVert_{H^p(\Omega_+)}^p+1,\]
  and $F(w)\in E^p(\Omega_+)$.
\end{proof}

$T$ is actually a bounded operator, and this fact will be proved right after  
Theorem~\ref{thm-170801-1650}. Before proving the inverse of 
the above proposition, we define the Blaschke product on $\Omega_+$, 
which is similar to that on $\mathbb{C}_+$, and introduce 
a factorization theorem of $H^p(\mathbb{C}_+)$.
 
\begin{lemma}[\cite{Ga07}]\label{lem-170801-1530}
  Let $\{z_n= x_n+\mathrm{i}y_n\}$ be a sequence of points in $\mathbb{C}_+$, 
  such that 
  \[\sum_{n=1}^{\infty}\frac{y_n}{1+|z_n|^2}< \infty,\]
  and $m$ be the number of $z_n$ equal to $\mathrm{i}$. 
  Then the Blaschke product
  \[B(z)= \bigg(\frac{z-\mathrm{i}}{z+\mathrm{i}}\bigg)^m
      \prod_{z_n\neq\mathrm{i}} \frac{|z_n^2+1|}{z_n^2+1}
      \cdot \frac{z-z_n}{z-\overline{z_n}}\]
  converges on $\mathbb{C}_+$, has non-tangential boundary limit $B(x)$ a.e.\@ 
  on $\mathbb{R}$, and the zeros of $B(z)$ are precisely 
  the points $z_n$, both counting multiplicity. Moreover, $|B(z)|<1$ 
  on $\mathbb{C}_+$ and $|B(x)|=1$ a.e.\@ on $\mathbb{R}$.
\end{lemma}

\begin{corollary}\label{lem-170801-1540}
  Let $\{w_n\}$ be a sequence of points in $\Omega_+$, such that 
  \[\sum_{n=1}^\infty \frac{\mathrm{Im}\,\Psi(w_n)}{1+|\Psi(w_n)|^2}
      < \infty,\]
  and $m$ be the number of $\Psi(w_n)$ equal to $\mathrm{i}$. 
  Then the Blaschke product
  \[B(w)= \bigg(\frac{\Psi(w)-\mathrm{i}}{\Psi(w)+\mathrm{i}}\bigg)^m
      \prod_{\Psi(w_n)\neq\mathrm{i}} \frac{|\Psi^2(w_n)+1|}{\Psi^2(w_n)+1}
      \cdot \frac{\Psi(w)-\Psi(w_n)}{\Psi(w)-\overline{\Psi(w_n)}},\]
  converges on $\Omega_+$, has non-tangential boundary limit $B(\zeta)$ a.e.\@ 
  on $\Gamma$, and the zeros of $B(w)$ are precisely 
  the points $w_n$, both counting multiplicity. Moreover, $|B(w)|<1$ 
  on $\Omega_+$ and $|B(\zeta)|=1$ a.e.\@ on $\Gamma$.
\end{corollary}

\begin{proof}
  This corollary of Lemma~\ref{lem-170801-1530} is obvious if we consider 
  the conformal mapping $w=\Phi(z)$ from $\mathbb{C}_+$ onto $\Omega_+$.
\end{proof}

\begin{lemma}[\cite{Ga07}]\label{lem-170801-1550}
  If $f(z)\in H^p(\mathbb{C}_+)$, $f\not\equiv 0$, $\{z_n\}$ are the zeros 
  of $f(z)$, and $B(z)$ is the Blaschke product associated with $\{z_n\}$, 
  Then
  \[g(z)=\frac{f(z)}{B(z)}\neq 0,\quad
    \text{and } \lVert g\rVert_{H^p(\mathbb{C}_+)}
      = \lVert f\rVert_{H^p(\mathbb{C}_+)}.\]
\end{lemma}

\begin{lemma}[\cite{DL17}]\label{lem-170801-1520}
  If $1< p< \infty$, then $H^p(\mathbb{C}_+)\subset T(H^p(\Omega_+))$ 
  or $T^{-1}(H^p(\mathbb{C}_+))\subset H^p(\Omega_+)$, and $T^{-1}$ 
  is bounded. Here, $T^{-1}f(w)= f(\Psi(w))(\Psi'(w))^{\frac1p}$, for 
  $w\in\Omega_+$ and $f(z)\in H^p(\mathbb{C}_+)$.
\end{lemma}

We could now prove that $T^{-1}$ is well-defined and bounded on 
$H^p(\mathbb{C}_+)$ for all $0<p<\infty$.
\begin{proposition}\label{pro-170801-1600}
  If $0< p< \infty$, then $H^p(\mathbb{C}_+)\subset T(H^p(\Omega_+))$ 
  or $T^{-1}(H^p(\mathbb{C}_+))\subset H^p(\Omega_+)$, and $T^{-1}$ 
    is bounded.
\end{proposition}

\begin{proof}
  We only need to prove this proposition under the assumption $0< p\leqslant 1$.
  Let $f(z)\in H^p(\mathbb{C}_+)$, then we could write $f(z)= B(z)g(z)$ 
  according to Lemma~\ref{lem-170801-1550}, where $g(z)$ has no zeros and 
  $\lVert g\rVert_{H^p(\mathbb{C}_+)}= \lVert f\rVert_{H^p(\mathbb{C}_+)}$. 
  Thus $g^\frac{p}2\in H^2(\mathbb{C}_+)$, and, by Lemma~\ref{lem-170801-1520}, 
  \[T_2^{-1}(g^{\frac{p}2})(w)
    = g^{\frac{p}2}(\Psi(w))(\Psi'(w))^{\frac12}
    \in H^2(\Omega_+),\]
  where $T_2^{-1}$ is the ``$p=2$'' version of $T^{-1}$. It follows that 
  \[g(\Psi(w))(\Psi'(w))^{\frac1p}\in H^p(\Omega_+).\] 
  Define $F(w)= f(\Psi(w))(\Psi'(w))^{\frac1p}$, then $TF(z)= f(z)$, and 
  \[F(w)= B(\Psi(w))g(\Psi(w))(\Psi'(w))^{\frac1p}.\]
  Since $|B(\Psi(w))|<1$ if $w\in\Omega_+$, we have $F(w)\in H^p(\Omega_+)$, 
  which shows that $H^p(\mathbb{C}_+)\subset T(H^p(\Omega_+))$. 
  Now $T^{-1}$ is well-defined on $H^p(\mathbb{C}_+)$, and 
  $T^{-1}f= F= f(\Psi)(\Psi')^{\frac1p}$.
  
  Lemma~\ref{lem-170801-1520} also implies that there exists a constant $C>0$, 
  such that, 
  \[\lVert T_2^{-1}(g^{\frac{p}2})\rVert_{H^2(\Omega_+)}
    = \lVert g^{\frac{p}2}(\Psi)(\Psi')^{\frac12}\rVert_{H^2(\Omega_+)}
    \leqslant C\lVert g^{\frac{p}2}\rVert_{H^2(\mathbb{C}_+)},\]
  then
  \begin{align*}
    \lVert F^p\rVert_{H^1(\Omega_+)}
    &\leqslant \lVert g^p(\Psi)\Psi'\rVert_{H^1(\Omega_+)}           \\
    &\leqslant C^2\lVert g^p\rVert_{H^1(\mathbb{C}_+)}
      = C^2\lVert f^p\rVert_{H^1(\mathbb{C}_+)},
  \end{align*}
  thus, 
  \[\lVert T^{-1}f\rVert_{H^p(\Omega_+)}
    = \lVert F\rVert_{H^p(\Omega_+)}
    \leqslant C^{\frac2p}\lVert f\rVert_{H^p(\mathbb{C}_+)},\]
  and we have proved that $T^{-1}$ is bounded.
\end{proof}

\begin{theorem}\label{thm-170801-1640}
  Let $F(w)\in H^p(\Omega_+)$, $F\not\equiv 0$, $\{w_n\}$ be the zeros 
  of $F(w)$, and $B(w)$ be the Blaschke product associated with $\{w_n\}$. 
  Then
  \[G(w)=\frac{F(w)}{B(w)}\in H^p(\Omega_+),\quad
    \text{and } \lVert F\rVert_{H^p(\Omega_+)}
    \leqslant \lVert G\rVert_{H^p(\Omega_+)}.\]  
\end{theorem}

\begin{proof}
  By Proposition~\ref{pro-170801-1410}, $TF(z)= F(\Phi(z)) 
    (\Phi'(z))^{\frac1p}\in H^p(\mathbb{C}_+)$, then $\{\Psi(w_n)\}$ is 
  exactly the zeros of $TF(z)$ since $\Phi'(z)\neq 0$. We know, 
  by Lemma~\ref{lem-170801-1550}, 
  \[TF(z)= B(z)g(z),\quad \text{and }
    \lVert g\rVert_{H^p(\mathbb{C}_+)}= \lVert TF\rVert_{H^p(\mathbb{C}_+)},\] 
  where $B(z)$ is the Blaschke product associated with $\{\Psi(w_n)\}$.
  
  Let $z=\Psi(w)$ for $w\in\Omega_+$, then
  $F(w)(\Psi'(w))^{-\frac1p}= B(\Psi(w))g(\Psi(w))$, or
  \[F(w)= B(\Psi(w))g(\Psi(w))(\Psi'(w))^{\frac1p}
    = B(\Psi(w))T^{-1}g(w),\]
  and $T^{-1}g(w)\in H^p(\Omega_+)$, by Proposition~\ref{pro-170801-1600}.
  We also know, by definition, that $B(\Psi(w))$ is the Blaschke product 
  associated with $\{w_n\}$.
  
  Since $|B(\Psi(w))|<1$ for $w\in\Omega_+$, it is obvious that 
  $\lVert F\rVert_{H^p(\Omega_+)}\leqslant \lVert G\rVert_{H^p(\Omega_+)}$, 
  and the theorem is proved.
\end{proof}

The proving method of the following theorem comes from \cite{De10}.

\begin{theorem}\label{thm-170801-1650}
  If $0<p<\infty$, $\tau>0$, and $F(w)\in H^p(\Omega_+)$, then $F(w)$ has 
  non-tangential boundary limit $F(\zeta)$, 
  which is in $L^p(\Gamma, |\mathrm{d}\zeta|)$, 
  a.e.\@ on $\Gamma$, and 
  \[\lim_{\tau\to 0} \lVert F(\cdot+\mathrm{i}\tau)- 
        F\rVert_{L^p(\Gamma,|\mathrm{d}\zeta|)}
    = 0.\]
\end{theorem}

\begin{proof}
  If $1<p<\infty$, then the existence of non-tangential boundary limit 
  was proved in \cite{DL17}, and the $L^p(\Gamma, |\mathrm{d}\zeta|)$ norm 
  convergence could be proved in the same way as in \cite{MC97}.
  
  We now assume $0< p\leqslant 1$, and write $F(w)=B(w)G(w)$ 
  by Theorem~\ref{thm-170801-1640}, 
  where $G(w)\in H^p(\Omega_+)$ has no zeros. 
  Then $G^{\frac{p}2}(w)\in H^2(\Omega_+)$ has non-tangential boundary limit 
  $G^{\frac{p}2}(\zeta)$ a.e.\@ on $\Gamma$, 
  and $G^{\frac{p}2}(\zeta)\in L^2(\Gamma,|\mathrm{d}\zeta|)$. 
  Togother with Corollary~\ref{lem-170801-1540}, we know that $F(w)$ has 
  non-tangential boundary limit $F(\zeta)= B(\zeta)G(\zeta)$ 
  a.e.\@ on $\Gamma$, and $F(\zeta)\in L^p(\Gamma,|\mathrm{d}\zeta|)$, 
  since $|F(\zeta)|= |G(\zeta)|$ a.e.\@ 
  and $G(\zeta)\in L^p(\Gamma,|\mathrm{d}\zeta|)$.
  
  The proof of the $L^p(\Gamma,|\mathrm{d}\zeta|)$ norm convergence 
  is more involved. Let $n$ be a positive integer such that $np>1$, 
  and denote $G^{\frac1n}(w)$ as $H(w)$, then $H(w)\in H^{np}(\Omega_+)$ 
  since $G(w)\in H^p(\Omega_+)$, and $H(w)$ has non-tangential boundary limit 
  $H(\zeta)= G^{\frac1p}(\zeta)$ a.e.\@ on $\Gamma$. If we let 
  $f_\tau(\zeta)= f(\zeta+\mathrm{i}\tau)$ for $f(w)$ defined on $\Omega_+$ 
  and $L^p(\Gamma)= L^p(\Gamma,|\mathrm{d}\zeta|)$ for simplicity, then
  \begin{equation}\label{equ-170801-1800}
    \lim_{\tau\to 0} \lVert H_\tau-H\rVert_{L^{np}(\Gamma)}= 0,\quad
    \text{and } \lVert H\rVert_{L^{np}(\Gamma)} 
        \leqslant \lVert H\rVert_{H^{np}(\Omega_+)}.
  \end{equation}
  The later comes from Fatou's lemma. Since $0< p\leqslant 1$, we have
  \begin{align*}
    \lVert F_\tau-F\rVert_{L^p(\Gamma)}^p
    &= \lVert B_\tau(G_\tau-G)+ (B_\tau-B)G\rVert_{L^p(\Gamma)}^p         \\
    &\leqslant \lVert B_\tau(G_\tau-G)\rVert_{L^p(\Gamma)}^p
         + \lVert (B_\tau-B)G\rVert_{L^p(\Gamma)}^p,       
  \end{align*}
  and then, as $|B_\tau|< 1$ and $G(\zeta)\in L^p(\Gamma)$,
  \begin{align*}
    \lim_{\tau\to 0} \lVert F_\tau-F\rVert_{L^p(\Gamma)}^p
    &\leqslant \lim_{\tau\to 0}\lVert B_\tau(G_\tau-G)\rVert_{L^p(\Gamma)}^p  \\
    &\leqslant \lim_{\tau\to 0} \lVert G_\tau-G\rVert_{L^p(\Gamma)}^p.
  \end{align*}
  The first inequality is by Lebesgue's dominated convergence theorem.
  
  Notice that, since $G(w)= H^n(w)$ and 
  $G(\zeta)= H^n(\zeta)$ a.e.\@ on $\Gamma$, 
  \[G_\tau- G= (H_\tau- H)\sum_{k=0}^{n-1} H_\tau^{n-1-k} H^k,\]
  we then apply H\"{o}lder's inequality, as $\frac1n+\frac{n-1}n=1$, 
  to obtain
  \begin{equation}\label{equ-170801-1810}
    \begin{aligned}
      \lVert G_\tau-G\rVert_{L^p(\Gamma)}^p
      &= \int_\Gamma |G_\tau- G|^p |\mathrm{d}\zeta|       \\
      &\leqslant \bigg(\int_\Gamma |H_\tau- H|^{np} 
              |\mathrm{d}\zeta|\bigg)^{\frac1n}
          \bigg(\int_\Gamma\Big|\sum_{k=0}^{n-1} H_\tau^{n-1-k} H^k
             \Big|^{\frac{np}{n-1}} |\mathrm{d}\zeta|\bigg)^{\frac{n-1}{n}}  \\
      &= \lVert H_\tau-H\rVert_{L^{np}(\Gamma)}^p
         \bigg\lVert \sum_{k=0}^{n-1} H_\tau^{n-1-k} H^k
             \bigg\rVert_{L^{\frac{np}{n-1}}(\Gamma)}^p.
    \end{aligned}
  \end{equation}
  Since $\frac{k}{n-1}+\frac{n-1-k}{n-1}=1$ for $0\leqslant k\leqslant n-1$, 
  we have
  \begin{align*}
    \int_\Gamma |H_\tau^{n-1-k} H^k|^{\frac{np}{n-1}} |\mathrm{d}\zeta| 
    &= \int_\Gamma |H_\tau|^{\frac{n-1-k}{n-1}np} 
           |H|^{\frac{k}{n-1}np} |\mathrm{d}\zeta|                   \\
    &\leqslant \bigg(\int_\Gamma |H_\tau|^{np} |\mathrm{d}\zeta|
            \bigg)^{\frac{n-1-k}{n-1}}
          \bigg(\int_\Gamma |H|^{np} |\mathrm{d}\zeta|
            \bigg)^{\frac{k}{n-1}}                                \\
    &\leqslant \big(\lVert H_\tau\rVert_{L^{np}(\Gamma)}^{np}
            \big)^{\frac{n-1-k}{n-1}}
          \big(\lVert H\rVert_{H^{np}(\Omega_+)}^{np}\big)^{\frac{k}{n-1}}   \\
    &\leqslant \lVert H\rVert_{H^{np}(\Omega_+)}^{np},
  \end{align*}
  which is 
  \[\lVert H_\tau^{n-1-k} H^k\rVert_{L^{\frac{np}{n-1}}(\Gamma)}
    \leqslant \lVert H\rVert_{H^{np}(\Omega_+)}^{n-1}.\]
  
  If $\frac{np}{n-1}>1$, then
  \begin{align*}
    \bigg\lVert \sum_{k=0}^{n-1} H_\tau^{n-1-k} H^k
       \bigg\rVert_{L^{\frac{np}{n-1}}(\Gamma)}
     &\leqslant \sum_{k=0}^{n-1} \lVert H_\tau^{n-1-k} H^k
       \rVert_{L^{\frac{np}{n-1}}(\Gamma)}                         \\
     &\leqslant n\lVert H\rVert_{H^{np}(\Omega_+)}^{n-1},
  \end{align*}
  or 
  \[\bigg\lVert \sum_{k=0}^{n-1} H_\tau^{n-1-k} H^k
       \bigg\rVert_{L^{\frac{np}{n-1}}(\Gamma)}^p
     \leqslant n^p \lVert H\rVert_{H^{np}(\Omega_+)}^{(n-1)p}.\]
  If $0< \frac{np}{n-1}\leqslant 1$, then 
  \begin{align*}
    \int_\Gamma\Big|\sum_{k=0}^{n-1} H_\tau^{n-1-k} H^k
       \Big|^{\frac{np}{n-1}} |\mathrm{d}\zeta|
    &\leqslant \sum_{k=0}^{n-1} \int_\Gamma |H_\tau^{n-1-k} H^k
           |^{\frac{np}{n-1}} |\mathrm{d}\zeta|                    \\
    &\leqslant n\lVert H\rVert_{H^{np}(\Omega_+)}^{np},
  \end{align*}
  or
  \[\bigg\lVert \sum_{k=0}^{n-1} H_\tau^{n-1-k} H^k
       \bigg\rVert_{L^{\frac{np}{n-1}}(\Gamma)}^p
     \leqslant n^{\frac{n-1}n} \lVert H\rVert_{H^{np}(\Omega_+)}^{(n-1)p}.\]
  Let $M_1=\max\{n^p,n^{\frac{n-1}n}\}$, then 
  inequality~\eqref{equ-170801-1810} becomes 
  \[\lVert G_\tau-G\rVert_{L^p(\Gamma)}^p
    \leqslant \lVert H_\tau-H\rVert_{L^{np}(\Gamma)}^p\cdot 
           M_1 \lVert H\rVert_{H^{np}(\Omega_+)}^{(n-1)p},\]
  which implies that, by \eqref{equ-170801-1800}, 
  \[\lim_{\tau\to 0} \lVert F_\tau-F\rVert_{L^p(\Gamma)}^p
    \leqslant \lim_{\tau\to 0} \lVert G_\tau-G\rVert_{L^p(\Gamma)}^p
    = 0,\]
  and we have proved the theorem.
\end{proof}

Now we are in the position of finishing the proof of the isomorphic theorem, 
that is, Theorem~\ref{thm-170803-1520}.

\begin{proof}[proof of Theorem~\ref{thm-170803-1520}]
  We only need to verify that $T$ is bounded in view of 
  Proposition~\ref{pro-170801-1410} and Proposition~\ref{pro-170801-1600}. 
  If $F(w)\in H^p(\Omega_+)$, then, by Theorem~\ref{thm-170801-1650}, 
  $F(w)$ has non-tangential boundary limit $F(\zeta)$ a.e.\@ on $\Gamma$. 
  By Fatou's lemma, we have 
  \begin{align*}
    \lVert TF\rVert_{H^p(\mathbb{C}_+)}^p
    &= \int_{\mathbb{R}} |F(\Phi)|^p |\Phi'|\,\mathrm{d}x   
     = \int_\Gamma |F|^p|\mathrm{d}\zeta|                             \\
    &= \int_\Gamma \liminf_{\tau\to 0} |F_\tau|^p|\mathrm{d}\zeta|    
     \leqslant \liminf_{\tau\to 0} \int_\Gamma |F_\tau|^p|\mathrm{d}\zeta|  \\
    &\leqslant \lVert F\rVert_{H^p(\Omega_+)}^p,
  \end{align*}
  then $\lVert TF\rVert_{H^p(\mathbb{C}_+)}
    \leqslant \lVert F\rVert_{H^p(\Omega_+)}$ 
  and $\lVert T\rVert\leqslant 1$.
\end{proof}

\section{More Results of $H^p(\Omega_+)$}

In this section, we will deduce, mainly from Theorem~\ref{thm-170803-1520} 
and Theorem~\ref{thm-170801-1650}, 
the Cauchy representation of $H^p(\Omega_+)$ while $1\leqslant p< \infty$, 
the completeness and separability of $H^p(\Omega_+)$ while $0<p<\infty$, 
and some other results which worth noticing. The following lemma 
was only mentioned in \cite{MC97}, and we write down the proof for clarity.
\begin{lemma}[\cite{MC97}]\label{lem-160628-0820}
  Let $0<p<\infty$, $F(w)$ be analytic on $\Omega_+$, and $\tau_1$, 
  $\tau_2$, $M_1$ are positive constants with $\tau_1<\tau_2$.
  If 
  \[\sup_{\tau\in[\tau_1,\tau_2]} \int_{\Gamma} |F(\zeta+\mathrm{i}\tau)|^p 
      |\mathrm{d}\zeta|\leqslant M_1\]
  then $F(\zeta(u)+\mathrm{i}\tau)\to 0$ uniformly for 
  $\tau_1+\delta\leqslant \tau\leqslant \tau_2-\delta$ as $|u|\to\infty$, 
  where $0<\delta<\frac{\tau_2-\tau_1}{2}$ is fixed.
\end{lemma}

\begin{proof}
  Since $|a'(u)|\leqslant M$, by the definition of $\Gamma$, if we let 
  $\rho=\frac{\delta}{\sqrt{1+M^2}}$, 
  $E_1=\{\zeta+\mathrm{i}\tau\colon \zeta\in\Gamma,
    \ \tau\in [\tau_1,\tau_2]\}$, 
  $E_2=\{\zeta+\mathrm{i}\tau\colon \zeta\in\Gamma,
    \ \tau\in [\tau_1+\delta,\tau_2-\delta]$, 
  then $\{w\colon |w-w_0|<\rho,\ w_0\in E_2\}\subset E_1$. 
  $|F(w)|^p$ is subharmonic on $\Omega_+$ as $F(w)$ is analytic on it, 
  and for all $w_0=u_0+\mathrm{i} v_0\in E_2$,
  \begin{align*}
    |F(w_0)|^p
    &\leqslant \frac{1}{\pi\rho^2} \iint_{|w-w_0|<\rho}
       |F(w)|^p \,\mathrm{d}\lambda(w)                             \\
    &\leqslant \frac{1}{\pi\rho^2} \int_{\tau_1}^{\tau_2}\!\!
       \int_{\{\zeta\in\Gamma\colon |\mathrm{Re}\,\zeta-u_0|<\rho\}}
         |F(\zeta+\mathrm{i}\tau)|^p |\mathrm{d}\zeta|\,\mathrm{d}\tau.
  \end{align*}
  
  Since $\int_{\Gamma} |F(\zeta+\mathrm{i}\tau)|^p |\mathrm{d}\zeta|$ 
  is bounded for $\tau\in[\tau_1,\tau_2]$, we know that 
  \[\frac{1}{\pi\rho^2} \int_{\tau_1}^{\tau_2}\!\!
     \int_{\Gamma} |F(\zeta+\mathrm{i}\tau)|^p
       |\mathrm{d}\zeta|\,\mathrm{d}\tau\] 
  is bounded. Since $\rho$ is fixed, by Lebesgue's dominated convergence theorem,
  \[\frac{1}{\pi\rho^2} \int_{\tau_1}^{\tau_2}\!\! 
      \int_{\{\zeta\in\Gamma\colon |\mathrm{Re}\,\zeta-u_0|<\rho\}}
      |F(\zeta+\mathrm{i}\tau)|^p |\mathrm{d}\zeta|\,\mathrm{d}\tau\to 0
      \text{ as } |u_0|\to\infty,\]
  and this shows that $F(w_0)\to 0$ uniformly as $|u_0|\to\infty$.
\end{proof}

\begin{lemma}\label{lem-170522-1043}
  If $1\leqslant p< \infty$ and $F(w)\in H^p(\Omega_+)$, then for each $\tau>0$ 
  and every $w_0=u_0+\mathrm{i}a(u_0)+\mathrm{i}\sigma$, 
  where $u_0$, $\sigma\in\mathbb{R}$,
  \[\frac{1}{2\pi\mathrm{i}}\int_{\Gamma_\tau}
      \frac{F(\zeta)}{\zeta-w_0}\,\mathrm{d}\zeta=\left\{\!\!
      \begin{array}{ll}
        F(w_0) & \text{if } \sigma>\tau,\\
        0 & \text{if } \sigma<\tau.
      \end{array}\right.\]
  Here, $\Gamma_{\tau}= \Gamma+\mathrm{i}\tau
    = \{\zeta+\mathrm{i}\tau\colon \zeta\in\Gamma\}$.
\end{lemma}

\begin{proof}
  If $p>1$, then the lemma was proved in \cite{MC97}, and we only need to 
  treat the $p=1$ case. Let $\tau'>\max\{\tau,\sigma\}$, $R>0$, and 
  $E=\{u+\mathrm{i}v\colon 
    -R<x <R,\ a(u)+\tau< y< a(u)+\tau'\}\subset\Omega_+$ 
  with its curve boundary $\partial E$ oriented such that 
  $E$ is on the ``left'' side of $\partial E$. 
  Denote $\partial E$ as $ABB'A'$, where
  \begin{align*}
    &AB= \{\zeta(u)+\mathrm{i}\tau\colon u\in[-R,R]\},\quad
     BB'= \{\zeta(R)+\mathrm{i}v\colon v\in[\tau,\tau']\}         \\
    &B'A'= \{\zeta(u)+\mathrm{i}\tau'\colon u\in[-R,R]\},\quad
     A'A= \{\zeta(-R)+\mathrm{i}v\colon v\in[\tau,\tau']\},
  \end{align*}
  then
  \[\frac{1}{2\pi\mathrm{i}}\int_{\partial E}
      \frac{F(\zeta)}{\zeta-w_0}\,\mathrm{d}\zeta=\left\{\!\!
      \begin{array}{ll}
        F(w_0) & \text{if } \sigma>\tau,\\
        0 & \text{if } \sigma<\tau,
      \end{array}
      \right.\]
  since $F(w)$ is analytic on $\Omega_+$.
  
  Let $M_R=\max\{|F(w)|\colon w\in BB'\}$, if we let $R\to\infty$, 
  then $M_R\to 0$ by lemma~\ref{lem-160628-0820}, and 
  \begin{align*}
  \Big|\int_{BB'} \frac{F(\zeta)}{\zeta-w_0}\,\mathrm{d}\zeta\Big|
  &\leqslant \int_{\tau}^{\tau'} \frac{M_R\,\mathrm{d} t}
      {|R+\mathrm{i} a(R)+\mathrm{i} t-w_0|}                  \\
  &\leqslant \int_{\tau}^{\tau'} \frac{M_R\,\mathrm{d} t}{|R-u_0|}
   = \frac{(\tau'-\tau)M_R}{|R-u_0|}
   \to 0.
  \end{align*}
  We also have $|\int_{AA'} F(\zeta)(\zeta-w_0)^{-1}\,\mathrm{d}\zeta|$  
  as $R\to \infty$. Since $\int_{\partial E}= 
    \int_{AB}+\int_{BB'}+\int_{B'A'}+\int_{A'A}$, 
  we get that
  \[\frac{1}{2\pi\mathrm{i}}\Big(
      \int_{\Gamma_{\tau}}-\int_{\Gamma_{\tau'}}\Big)
      \frac{F(\zeta)}{\zeta-w_0}\,\mathrm{d}\zeta=\left\{\!\!\!
      \begin{array}{ll}
        F(w_0) & \text{if } \sigma>\tau,\\
        0 & \text{if } \sigma<\tau.
      \end{array}
      \right.\]
      
  We know that
  \[\int_{\Gamma_{\tau'}}|F(\zeta)||\mathrm{d}\zeta|
    = \int_{\Gamma}|F(\zeta+\mathrm{i}\tau')||\mathrm{d}\zeta|
    \leqslant \lVert F\rVert_{H^1(\Omega_+)},\]
  and 
  \[\bigg|\int_{\Gamma_{\tau'}} 
      \frac{F(\zeta)}{\zeta-w_0}\,\mathrm{d}\zeta\bigg|
    \leqslant \int_{\Gamma_{\tau'}} |F(\zeta)|\,|\mathrm{d}\zeta|
      \cdot \sup_{\zeta\in \Gamma_{\tau'}}\frac1{|\zeta-w_0|}
    \leqslant \lVert F\rVert_{H^1(\Omega_+)}
      \frac{\sqrt{1+M^2}}{\tau'-\sigma}.\] 
  Now we have  
  \[\int_{\Gamma_{\tau'}} \frac{F(\zeta)}{\zeta-w_0}\,\mathrm{d}\zeta
      \to 0 \text{ as } \tau'\to\infty,\]
  and thus finishes the proof of the lemma.
\end{proof}

\begin{theorem}\label{thm-170803-1020}
  If $1\leqslant p< \infty$, then $F(w)\in H^p(\Omega_+)$ is the Cauchy integral 
  of its non-tangential boundary limit $F(\zeta)$, that is 
  \[F(w)= \frac1{2\pi\mathrm{i}} \int_\Gamma 
      \frac{F(\zeta)}{\zeta-w}\,\mathrm{d}\zeta,
    \quad\text{for } w\in\Omega_+.\]
  And we also have 
  \[0= \frac1{2\pi\mathrm{i}} \int_\Gamma 
        \frac{F(\zeta)}{\zeta-w}\,\mathrm{d}\zeta,
      \quad\text{for } w\in\Omega_-.\]
\end{theorem}

\begin{proof}
  The case of $p>1$ was proved by using the same method as in \cite{MC97}, 
  and we shall let $F(w)\in H^1(\Omega_+)$.
  For fixed $w_0\in\Omega_\pm$, we could write $w_0=\zeta_0+\mathrm{i\sigma}$, 
  where $\zeta_0\in\Gamma$ and $\sigma\neq 0$, then 
  if $0<\tau<\frac{|\sigma|}2$, we have, by Lemma~\ref{lem-170522-1043},
  \[\frac1{2\pi\mathrm{i}} \int_\Gamma 
      \frac{F(\zeta+\mathrm{i}\tau)\,\mathrm{d}\zeta}
           {\zeta+\mathrm{i}\tau-w_0}
    =\left\{\!\!
          \begin{array}{ll}
            F(w_0) & \text{if $\sigma>0$ or $w_0\in\Omega_+$},\\
            0 & \text{if $\sigma<0$ or $w_0\in\Omega_-$},
          \end{array}\right.\]
  then
  \begin{align*}
    I
    &= \bigg|\frac1{2\pi\mathrm{i}} \int_\Gamma 
          \frac{F(\zeta+\mathrm{i}\tau)\,\mathrm{d}\zeta}
               {\zeta+\mathrm{i}\tau-w_0}
          - \frac1{2\pi\mathrm{i}} \int_\Gamma 
            \frac{F(\zeta)\,\mathrm{d}\zeta}{\zeta-w_0}\bigg|       \\
    &\leqslant \frac1{2\pi} \int_\Gamma 
          \frac{|F(\zeta+\mathrm{i}\tau)-F(\zeta)|}
               {|\zeta+\mathrm{i}\tau-w_0|} |\mathrm{d}\zeta|      \\
    &{\phantom{\leqslant{}}}{}+ \frac1{2\pi} \int_\Gamma 
         \Big|\frac1{\zeta+\mathrm{i}\tau-w_0}- \frac1{\zeta-w_0}\Big| 
         |F(\zeta)\,\mathrm{d}\zeta|                           \\
    &= I_1+ I_2.
  \end{align*}
  We get $|\zeta+\mathrm{i}\tau-w_0|\geqslant \frac{|\sigma|}{2\sqrt{1+M^2}}$ 
  from $0<\tau<\frac{|\sigma|}2$, and by Theorem~\ref{thm-170801-1650},
  \[I_1\leqslant \frac1{2\pi}\cdot\frac{2\sqrt{1+M^2}}{|\sigma|}
        \lVert{F(\cdot+\mathrm{i}\tau)-F}\rVert_{L^1(\Gamma,|\mathrm{d}\zeta|)}
    \to 0 \text{ as } \tau\to 0.\]
  Since $|\zeta-w_0|\geqslant \frac{|\sigma|}{\sqrt{1+M^2}}$ and 
  $|\zeta+\mathrm{i}\tau-w_0|\geqslant \frac{|\sigma|}{2\sqrt{1+M^2}}$, 
  the integrand in $I_2$ is bounded by 
  $\frac{3\sqrt{1+M^2}}{|\sigma|}|F(\zeta)|$, 
  which is integrable as $F(\zeta)\in L^1(\Gamma, |\mathrm{d}\zeta|)$, 
  thus $I_2\to 0$ as $\tau\to 0$ by Lebesgue's dominated convergence theorem.
  
  Then we know that $I$ could be as small as we wish, which means that
  \[\frac1{2\pi\mathrm{i}} \int_\Gamma 
      \frac{F(\zeta)}{\zeta-w_0}\,\mathrm{d}\zeta
    =\left\{\!\!
        \begin{array}{ll}
          F(w_0) & \text{if $w_0\in\Omega_+$},\\
          0 & \text{if $w_0\in\Omega_-$},
        \end{array}\right.\]
  and the theorem is proved.
\end{proof}

\begin{theorem}\label{thm-170803-1110}
  If $1\leqslant p< \infty$, $F(w)\in H^p(\Omega_+)$, then 
  \[F(w)= \int_\Gamma K_{\mathrm{i}\tau}(\zeta,\zeta_0) 
            F(\zeta)\,\mathrm{d}\zeta,\]
  for $w= \zeta_0+\mathrm{i}\tau$ with $\zeta_0\in\Gamma$ and $\tau>0$.
\end{theorem}

\begin{proof}
  Let $w'= \zeta_0-\mathrm{i}\tau$, then $w'\in\Omega_-$. 
  By Theorem~\ref{thm-170803-1020}, we have
  \[F(w)= \frac1{2\pi\mathrm{i}} \int_\Gamma 
          \frac{F(\zeta)}{\zeta-w}\,\mathrm{d}\zeta,\quad
    \text{and } 
    0= \frac1{2\pi\mathrm{i}} \int_\Gamma 
          \frac{F(\zeta)}{\zeta-w'}\,\mathrm{d}\zeta.\]
  Subtract the second equation from the first one, then 
  \begin{align*}
    F(w)
    &= \frac1{2\pi\mathrm{i}} \int_\Gamma F(\zeta) \Big(
         \frac{1}{\zeta-w}- \frac{1}{\zeta-w'}\Big)\mathrm{d}\zeta    \\
    &= \int_\Gamma K_{\mathrm{i}\tau}(\zeta,\zeta_0) 
          F(\zeta)\,\mathrm{d}\zeta,
  \end{align*}
  by the definition of $K_z(\zeta,\zeta_0)$.
\end{proof}

We then introduce some lemmas about $K_z(\zeta,\zeta_0)$.
\begin{lemma}[\cite{DL17}]\label{lem-170622-2050}
  There exists a constant $C>0$, depending on $M$, such that
  \[|K_{\mathrm{i}\tau}(\zeta,\zeta_0)|
    \leqslant \frac{C\tau}{|\zeta-\zeta_0|^2+\tau^2},\quad
    \text{for } \tau\in\mathbb{R}.\]
\end{lemma}

\begin{corollary}\label{cor-170803-1040}
  If $1\leqslant p< \infty$, $F(\zeta)\in L^p(\Gamma,|\mathrm{d}\zeta|)$, 
  and we define
  \[F(w)
  = \int_\Gamma K_{\mathrm{i}\tau}(\zeta,\zeta_0) F(\zeta)
  \,\mathrm{d}\zeta,\]
  for $w=\zeta_0+\mathrm{i}\tau\in\Omega_+$, 
  where $\zeta_0\in\Gamma$ and $\tau>0$, then 
  \[\sup_{\tau>0} \int_{\Gamma_\tau} |F(w)|^p |\mathrm{d}w|< \infty.\]
\end{corollary}

\begin{proof}
  The case of $p>1$ was proved in \cite{DL17}. If $p=1$, the lemma is similarly proved.
\end{proof}

\begin{corollary}
  If $0<p<q$, $F(w)\in H^p(\Omega_+)$, 
  and $F(\zeta)\in L^q(\Gamma,|\mathrm{d}\zeta|)$, 
  then $F(w)\in H^q(\Omega_+)$.
\end{corollary}

\begin{proof}
  For $F(w)\in H^p(\Omega_+)$, we write, by Theorem~\ref{thm-170801-1640}, 
  $F(w)=B(w)G(w)$, where $|B(\zeta)|=1$ a.e.\@ on $\Gamma$, 
  $G(w)\neq 0$ and is in $H^p(\Omega_+)$. Choose a positive integer~$n$ 
  such that $np>1$, and define $H(w)= G^{\frac1n}(w)$, 
  then $H(w)\in H^{np}(\Omega_+)$, $H(\zeta)$ exists a.e.\@ on $\Gamma$.
  By Theorem~\ref{thm-170803-1110}, 
  for $w= \zeta_0+\mathrm{i}\tau\in \Omega_+$, we have
  \[H(w)
    = \int_\Gamma K_{\mathrm{i}\tau}(\zeta,\zeta_0) H(\zeta)\,\mathrm{d}\zeta.\]
  Since $F(\zeta)= B(\zeta)G(\zeta)= B(\zeta)H^n(\zeta)$ 
  and $F(\zeta)\in L^q(\Gamma,|\mathrm{d}\zeta|)$, 
  we have $H(\zeta)\in L^{nq}(\Gamma,|\mathrm{d}\zeta|)$, 
  thus $H(w)\in H^{nq}(\Omega_+)$, by Corollary~\ref{cor-170803-1040}, 
  as $nq>np>1$. Now we could deduce, from 
  \[|F(w)|= |B(w)G(w)|= |B(w)H^n(w)|\leqslant |H(w)|^n,\]
  that $F(w)\in H^q(\Omega_+)$, which finishes the proof of the corollary.
\end{proof}

\begin{lemma}[\cite{DL17}]\label{lem-170629-2230}
  Suppose $\zeta_0=\zeta(u_0)\in\Gamma$, $\phi\in(0,\frac\pi2)$ are both fixed, 
  and $\zeta'(u_0)$ exists, then we could choose positive constants 
  $C$ and $\delta$, depending on $\phi$ and $\zeta_0$, respectively, 
  such that if $z+\zeta\in\Omega_\phi(\zeta_0)\cap\Omega_+$ 
  and $|z|<\delta$, then
  \[|K_z(\zeta,\zeta_0)|
  \leqslant \frac{C|z|}{|\zeta-\zeta_0|^2+|z|^2}.\]
\end{lemma}

\begin{corollary}\label{cor-170803-1050}
  If $1\leqslant p< \infty$, $F(\zeta)\in L^p(\Gamma,|\mathrm{d}\zeta|)$, 
  and $u_0$ is the Lebesgue point of $F(u+\mathrm{i}a(u))$ such that 
  $\zeta'(u_0)= |\zeta'(u_0)|\mathrm{e}^{\mathrm{i}\phi_0}$ exists,
  where $\phi_0\in(-\frac\pi2,\frac\pi2)$, 
  then for any $\phi\in(0,\frac{\pi}2)$, we have
  \[\lim_{\substack{z+\zeta_0\in\Omega_\phi(\zeta_0)\cap\Omega_+,\\ 
      z\to 0}}
  \int_\Gamma K_z(\zeta,\zeta_0)F(\zeta)\,\mathrm{d}\zeta
  = F(\zeta_0).\]
\end{corollary}
\begin{proof}
  If $p>1$, then the corollary has been proved in \cite{DL17}. 
  If $p=1$, the proof is nearly the same, thus we omit it here.
\end{proof}

The above corollary shows that, for $w=\zeta_0+z\in\Omega_+$ 
where $\zeta_0\in\Gamma$ and $z\in\mathrm{C}$, if we define 
$G(w)= G(\zeta_0+z)
=\int_\Gamma K_z(\zeta,\zeta_0)F(\zeta)\,\mathrm{d}\zeta$, 
then $G(w)$ has non-tangential boundary limit $F(\zeta_0)$ at $\zeta_0$,
although $G(w)$ maybe only defined on $\Omega_\phi(\zeta_0)$ 
and near $\zeta_0$.

Notice that if $G(w)\in H^\infty(\Omega_+)$, then it has non-tangential 
boundary limit $G(\zeta)$ a.e.\@ on $\Gamma$, since 
$G(\Phi(z))\in H^\infty(\mathbb{C}_+)$ has non-tangential boundary limit 
a.e.\@ on the real axis.

\begin{lemma}\label{lem-170716-1340}
  If $1\leqslant p< \infty$, $\frac1p+\frac1q=1$, $F(w)\in H^p(\Omega_+)$ and 
  $G(w)\in H^q(\Omega_+)$, then
  \[\int_\Gamma F(\zeta)G(\zeta)\,\mathrm{d}\zeta= 0.\]
\end{lemma}

\begin{proof}
  If $p>1$, Meyer and Coifman have proved this lemma in \cite{MC97}. 
  We now assume $F\in H^1(\Omega_+)$ and prove
  $\int_\Gamma F(\zeta)\,\mathrm{d}\zeta= 0$ first.
  By using the same method as in Lemma~\ref{lem-170522-1043}, 
  we know that the value of 
  $\int_{\Gamma} F(\zeta+\mathrm{i}\tau)\,\mathrm{d}\zeta$ 
  is independent of $\tau$ for $\tau>0$.
  Choose $w_0=\zeta(u_0)+\mathrm{i}\sigma$, where $\sigma<\tau$, 
  then by Lemma~\ref{lem-170522-1043},
  \[\int_\Gamma \frac{F(\zeta+\mathrm{i}\tau)}{\zeta+\mathrm{i}\tau-w_0}
        \,\mathrm{d}\zeta= 0,\]
  that is 
  \[\int_\Gamma \frac{(\tau-\sigma)F(\zeta+\mathrm{i}\tau)}{
      \zeta+\mathrm{i}\tau-w_0}\,\mathrm{d}\zeta= 0.\]
  Let $\zeta=u+\mathrm{i} a(u)$ and the denominator of 
  the integrand above be $h(\zeta)$, then by the definition of $a(u)$, 
  \begin{align*}
    |h(\zeta)|^2
    &= |u-u_0|^2+\big(a(u)-a(u_0)+\tau-\sigma\big)^2    \\
    &\geqslant \frac{1+M^2}{M^2}\big(a(u)-a(u_0)\big)^2
        + 2(\tau-\sigma)\big(a(u)-a(u_0)\big)
        + (\tau-\sigma)^2                           \\
    &\geqslant \frac{(\tau-\sigma)^2}{1+M^2},
  \end{align*}
  that is, the integrand is dominated by 
  $\sqrt{1+M^2}|F(\zeta+\mathrm{i}\tau)|$ which is 
  in $L^1(\Gamma, |\mathrm{d}\zeta|)$.
  By Lebesgue's dominated convergence theorem, if we let 
  $\sigma\to -\infty$ and notice that $w_0=\zeta(u_0)+\mathrm{i}\sigma$, 
  then
  \[\int_{\Gamma} F(\zeta+\mathrm{i}\tau)\,\mathrm{d}\zeta
    = \lim_{\sigma\to -\infty} \int_\Gamma 
          \frac{(\tau-\sigma)F(\zeta+\mathrm{i}\tau)}
               {\zeta+\mathrm{i}\tau-w_0}\,\mathrm{d}\zeta
    = 0.\]
  Now $\int_{\Gamma} F(z)\,\mathrm{d} z=0$ is obvious if we notice, 
  by Theorem~\ref{thm-170801-1650}, that
  \[\int_{\Gamma} |F(\zeta+\mathrm{i}\tau)-F(\zeta)|
        |\mathrm{d}\zeta| =0 \text{ as } \tau\to\infty.\]
  
  If $F\in H^1(\Omega_+)$ and $G\in H^{\infty}(\Omega_+)$, then 
  $FG\in H^1(\Omega_+)$, thus
  \[\int_\Gamma F(\zeta)G(\zeta)\,\mathrm{d}\zeta=0,\]
  and the lemma is proved.
\end{proof}

The next theorem gives the sufficient and necessary condition of 
a $L^p(\Gamma,|\mathrm{d}\zeta|)$ function be the non-tangential boundary  
of a function in $H^p(\Omega_+)$, where $1\leqslant p< \infty$.
\begin{theorem}\label{thm-170803-1120}
  If $1\leqslant p< \infty$, and $F(\zeta)\in L^p(\Gamma,\lvert\mathrm{d}\zeta\rvert)$, 
  then $F(\zeta)$ is the non-tangential boundary limit of a function 
  in $H^p(\Omega_+)$ if and only if
  \[\int_{\Gamma}\frac{F(\zeta)}{\zeta-\alpha}\,\mathrm{d}\zeta=0, \quad
  \text{for all } \alpha\notin\overline{\Omega_+}.\]
\end{theorem}

\begin{proof}
  The ``$p>1$'' version of this theorem was proved in \cite{DL17}, and we now 
  assume $p=1$.
  
  ``$\Rightarrow$'': for fixed $\alpha\notin\overline{\Omega_+}$, 
  $G(w)=\frac1{w-\alpha}\in H^{\infty}(\Omega_+)$ and its non-tangential 
  boundary limit is $G(\zeta)= \frac1{\zeta-\alpha}$, then by 
  Lemma~\ref{lem-170716-1340}, we have 
  \[\int_{\Gamma}\frac{F(\zeta)}{\zeta-\alpha}\,\mathrm{d}\zeta=0.\]
  
  ``$\Leftarrow$'': for $w\in\Omega_+$, define 
  \[F(w)= \frac1{2\pi\mathrm{i}} 
  \int_{\Gamma}\frac{F(\zeta)}{\zeta-w}\,\mathrm{d}\zeta.\]
  For fixed $w_1\in\Omega_+$, there exists a constant $\delta>0$ such that 
  the open disk $D(w_1,2\delta)\subset\Omega_+$. 
  Choose $w_2\in D(w_1,\delta)$, then for $\zeta\in\Gamma$,
  \[|\zeta-w_2|\geqslant |\zeta-w_1|\geqslant \delta,\]
  and
  \begin{align*}
    |F(w_1)-F(w_2)|
    &\leqslant \frac1{2\pi} \int_\Gamma 
          \frac{|(w_1-w_2)F(\zeta)|}{|\zeta-w_1||\zeta-w_2|}
          |\mathrm{d}\zeta|                                   \\
    &\leqslant \frac{|w_1-w_2|}{2\pi\cdot 2\delta^2} 
        \int_\Gamma |F(\zeta)||\mathrm{d}\zeta|            \\
    &\leqslant \frac{1}{4\pi\delta^2} 
        \lVert F\rVert_{L^1(\Gamma,|\mathrm{d}\zeta|)}|w_1-w_2|,
  \end{align*}
  which shows that $F(w)$ is continuous on $\Omega_+$. Now, 
  it is easy to deduce from Morera's theorem that $F(w)$ is also analytic.
  
  If we write $w=\zeta_0+\mathrm{i}\tau$ where $\zeta_0\in\Gamma$ 
  and $\tau>0$, then $\zeta_0-\mathrm{i}\tau\in\Omega_-$, and
  \begin{align*}
    F(w)
    &= \frac1{2\pi\mathrm{i}} \int_{\Gamma} F(\zeta) \Big(
    \frac1{\zeta-(\zeta_0+\mathrm{i}\tau)}
    - \frac1{\zeta-(\zeta_0-\mathrm{i}\tau)}\Big)\mathrm{d}\zeta   \\
    &= \int_\Gamma F(\zeta)K_{\mathrm{i}\tau}(\zeta,\zeta_0) 
    \,\mathrm{d}\zeta.
  \end{align*}
  By Corollary~\ref{cor-170803-1040}, $F(w)\in H^1(\Omega_+)$.
  
  For fixed $\zeta_0\in\Gamma$ and $\phi\in(0,\frac{\pi}2)$, if 
  $w\in\Omega_\phi(\zeta_0)\cap\Omega_+$, 
  we write $w=\zeta_0+z$, then there exists $\delta>0$, such that 
  $w_0-z\in\Omega_-$ for all $|z|<\delta$, and
  \begin{align*}
    F(w)
    &= \frac1{2\pi\mathrm{i}} \int_{\Gamma} F(\zeta) \Big(
    \frac1{\zeta-(\zeta_0+z)}- \frac1{\zeta-(\zeta_0-z)}\Big)
    \mathrm{d}\zeta   \\
    &= \int_\Gamma F(\zeta)K_z(\zeta,\zeta_0) 
    \,\mathrm{d}\zeta.
  \end{align*}
  By Corollary~\ref{cor-170803-1050}, $F(w)\to F(\zeta_0)$ if $w\to\zeta_0$, 
  that is $F(w)$ has non-tangential boudary limit $F(\zeta_0)$ 
  at $\zeta_0\in\Gamma$.
  
  Thus, $F(\zeta)$ is the non-tangential boundary limit function of
  $F(w)\in H^1(\Omega_+)$.
\end{proof}

\begin{lemma}[\cite{De10}]
  If $0<p<\infty$, then $H^p(\mathbb{C}_+)$ is complete and seperable.
\end{lemma}

\begin{theorem}
  If $0<p<\infty$, then $H^p(\Omega_+)$ is complete and seperable.
\end{theorem}

\begin{proof}
  Suppose $\{F_n(w)\}$ is a Cauchy sequence in $H^p(\Omega_+)$, that is 
  \[\lVert F_n-F_m\rVert_{H^p(\Omega_+)}\to 0,\quad
  \text{as }n,m\to\infty.\]
  By Theorem~\ref{thm-170803-1520}, 
  \[\lVert TF_n-TF_m\rVert_{H^p(\mathbb{C}_+)} 
    \leqslant \lVert F_n-F_m\rVert_{H^p(\Omega_+)},\]
  then $\{TF_n(z)\}$ is a Cauchy sequence in $H^p(\mathbb{C}_+)$, 
  and we suppose it converges to $f(z)\in H^p(\mathbb{C}_+)$. 
  Since $T^{-1}f\in H^p(\Omega_+)$, and 
  \[\lVert F_n-T^{-1}f\rVert_{H^p(\Omega_+)}
    \leqslant \lVert T^{-1}\rVert\lVert TF_n-f\rVert_{H^p(\mathbb{C}_+)},\]
  we know that $\{F_n\}$ converges in $H^p(\Omega_+)$, 
  thus $H^p(\Omega_+)$ is complete. 
  The separability could be similarly proved.
\end{proof}

\section*{Funding}
This work is supported by National Natural Science Foundation 
of China(Grant No.\@ 11271045).


\begin{thebibliography}{99}
\bibitem{Du70} Duren PL. Theory of $H^p$ Spaces. 
  New York: Academic Press; 1970.

\bibitem{DL17} Deng GT, Liu R. Hardy Spaces ($1\leqslant p< \infty$) 
  over Lipschitz Domains. arXiv:1708.01188 [math.CV].  

\bibitem{De10} Deng GT. Complex Analysis (in Chinese), 
  Beijing: Beijing Normal University Press; 2010.

\bibitem{Ke90} Kenig C. Weighted $H^p$ spaces on Lipschitz domains. 
  Amer. J. Math. 1980;102:129-163.
  
\bibitem{Ga07} Garnett JB. Bounded Analytic Functions, 
  New York: Springer; 2007.
    
\bibitem{MC97} Meyer Y, Coifman R. 
  Wavelets: Calder\'on-Zygmund and Multilinear Operators. 
  Cambridge (UK): Cambridge University Press; 1997.

\end{thebibliography}
\end{document}